\documentclass{amsart}
\usepackage{amsmath}
\usepackage{amsthm}
\usepackage{amssymb}
\usepackage{exscale}
\usepackage{latexsym}

 \newtheorem{thm}[equation]{Theorem}
\newtheorem{theorem}[equation]{Theorem}
 
 \newtheorem{prop}[equation]{Proposition}
 
 \newtheorem{cor}[equation]{Corollary}
 \newtheorem{lemma}[equation]{Lemma}
 \theoremstyle{definition}

 \newtheorem{remark}[equation]{Remark}

 \DeclareMathOperator{\Gl}{Gl}
 \DeclareMathOperator{\Sl}{Sl}

 \newcommand{\tens} {\otimes}
 
 \newcommand{\field}[1] {\mathbb #1}

 \newcommand{\Z} {\field Z}
 \newcommand{\C} {\field C}
 
 \newcommand{\F} {{\field F\:\!}}
 
 \newcommand{\A}{{\mathcal A}}
\newcommand{\om}{\omega}

\newcommand{\codim}{\text{codim}}

\newcommand{\vol}{\text{\it vol}}
\newcommand{\Om}{\Omega}

\begin{document}

\title[Reflection groups and Differential Forms]
{Reflection Groups and Differential Forms}

\author[Julia Hartmann]{Julia Hartmann}
\address{IWR\\ University of Heidelberg\\Im Neuenheimer Feld 368, 69121 Heidelberg, Germany}
\email{julia.hartmann@iwr.uni-heidelberg.de}
\author[Anne V. Shepler]{Anne V. Shepler}
\address{Dept.\ of Mathematics\\
         University of North Texas\\
         P.O. Box 311430,
         Denton, TX, 76203}
\email{ashepler@unt.edu}

\thanks{Work of first author partially supported by DFG 
(German National Science Foundation).
Work of second author partially supported by National Science Foundation grant \#DMS-0402819 and National Security Agency grant MDA904-03-1-0005.}

\begin{abstract}
We study differential forms invariant under a finite reflection group over a field of arbitrary characteristic. In particular, we prove an analogue of Saito's freeness criterion for invariant differential 1-forms. We also discuss how twisted wedging endows the invariant forms with the structure of a free exterior algebra in certain cases.
Some of the results are extended to the case of relative invariants with respect to a linear character.

\end{abstract}

\maketitle

\section{Introduction}
\label{introduction}

\vspace{3ex}

The classical study of reflection groups (in characteristic zero) 
uses the theory of hyperplane arrangements to exhibit natural structures 
on the sets of invariant polynomials, derivations, and differential forms.
In this article, we begin a theory 
linking hyperplane arrangements and invariant forms
for reflection groups over arbitrary fields.

Let
 $V$ be an $n$-dimensional vector space over a field $\F$,
and let $G\leq \Gl_n(\F)$ be a finite group.
Let $\F[V]= S(V^*)$, the ring of polynomials on $V$. 
We consider the $\F[V]$-modules of differential $k$-forms,
$$
 \Om^k  := \F[V]\otimes \Lambda^k (V^*),
$$
and more generally, the module of differential forms, $\Om := \bigoplus_k \ \Om^k$. 
We are interested in the invariants under the action of $G$ on these modules.
Invariant (and relatively invariant) differential forms have applications to various areas of mathematics, for example, dynamical systems
(see  \cite{DoyleMcMullen} and \cite{CrassDoyle}), group cohomology (see \cite{Adem}),
symplectic reflection algebras and Hecke algebras (see
\cite{EtingofGinzburg} and \cite{SheplerWitherspoon}),
topology of complement spaces (see \cite{orliksolomon} and \cite{shepler05}), and Riemannian manifolds (see \cite{MichorI}).

We determine the rank of $(\Om^1)^G$ over an arbitrary field
in Theorem~\ref{rank}.
We then restrict our attention to the case when $G$ is generated by reflections.
When the characteristic of $\F$ is coprime to the group order (the {\bf nonmodular} case),
the ring of invariant polynomials $\F[V]^G$ forms a polynomial algebra:
$\F[V]^G=\F[f_1,\ldots,f_n]$ for some algebraically independent polynomials $f_i$ in $\F[V]$.
We say that $G$ has {\bf polynomial invariants} if $\F[V]^G$ has this form and we call
the polynomials $f_i$ {\bf basic invariants}.
Solomon~\cite{solomon} showed that 
the exterior derivatives $df_1, \ldots, df_n$
generate the set of invariant differential forms $\Om^G$
as an exterior algebra 
in the nonmodular setting:
$$\Om^G=\F[f_1,\ldots,f_n]\otimes\Lambda(df_1,\ldots,df_n).$$  

In the modular case, when the characteristic of $\F$ divides the group order,
a reflection group may fail to have polynomial invariants.  Even when the invariants
$\F[V]^G$ do form a polynomial algebra, the exterior derivatives $df_1, \ldots,
df_n$ of basic invariants $f_1, \ldots, f_n$ may fail to generate $(\Om^1)^G$ as an
$\F[V]^G$-module, in which case they will certainly not generate $\Om^G$ as an
algebra under the usual wedging of forms (see Section~\ref{examples} for an example). 
Hartmann \cite{hartmann} showed that Solomon's theorem holds for groups with
polynomial invariants 
if and only if the group $G$ 
contains no transvections.  

We investigate the invariant theory in the general case (when $G$ may contain
transvections) and explore two fundamental questions:
\begin{itemize}
\item[(1)]
When is the module of invariant 1-forms  free over the invariant ring $\F[V]^G$?
\item[(2)]
When is the module of invariant forms a free exterior algebra, 
i.e., 
when is $\Om^G=\F[V]^G\otimes\Lambda(\om_1,\ldots,\om_n)$ for some
1-forms $\om_i$?
\end{itemize}
In Theorem~\ref{solomon}, we address the first question by giving a criterion
for when a set of invariant 1-forms generates $(\Om^1)^G$ as a free module over 
 $\F[V]^G$.
In Theorem~\ref{freealgebra}, we show that a maximality condition on the root spaces allows one
to endow $\Om^G$ with the structure of a free exterior algebra.
The exterior algebra structure emerges from  a twisted wedging operator 
introduced (for nonmodular groups) by Shepler \cite{shepler}. (Beck \cite{beck} uses twisted wedging to extend results in a different direction 
in the nonmodular case.) 
In Section~\ref{chi}, we generalize Theorem~\ref{solomon} to the case of forms which are invariant with respect to a linear character of the group $G$.  

We include a special analysis for three classes of reflection groups:
groups fixing a single hyperplane pointwise,
groups containing the special linear group,
and unipotent groups. In all these cases,
we observe that $(\Om^1)^G$ is free as an $\F[V]^G$-module.   
Furthermore, our examples suggest a strategy for producing generators from the
exterior derivatives $df_1, \ldots, df_n$ of a choice of basic invariants
$f_1, \ldots, f_n$ for $G$, and this strategy 
is related to the geometry of the reflecting hyperplanes (see
Section~\ref{examples}). It is therefore natural to ask whether $(\Om^1)^G$ is
{\em always} a free $\F[V]^G$-module when $G$ is a reflection group.  
This question remains open.

\vspace{1ex}

{\bf Note:} Since the set of differential forms in characteristic $2$ is a truncated polynomial algebra, we exclude that case from our considerations throughout.

\section{The rank of invariant differential forms}
Before restricting to the case of reflection groups, we compute the rank of
$(\Om^1)^G$ when $G\leq \Gl_n(\F)$ is an arbitrary finite group. 
Recall that the rank of a module $M$ over an
integral domain $R$ is defined as the maximal number of $R$-linearly independent
elements of $M$. It equals the dimension of $F\tens_R M$, where $F$ is the field
of fractions of $R$.

\begin{thm}\label{rank}
Let $\F$ be a field and let $G\leq \Gl_n(\F)$ be a finite group. Then
$(\Om^1)^G$ has rank $n$.
\end{thm}
\begin{proof}
First, we show that the rank is at most $n$. To this end, suppose that
$\om_1,\ldots,\om_m$ are $\F[V]^G$-linearly independent $1$-forms. Since $G$ is
finite, the fraction field of $\F[V]^G$ equals $\F(V)^G$, the field of invariant
rational functions on $V$ (see, for example, \cite{Poly},
Prop.~1.2.4). Consequently, $\om_1,\ldots,\om_m$ are linearly independent over
$\F(V)^G$. We claim they are in fact linearly independent over $\F(V)$, thus
forcing $m\leq n$. Assume the contrary and consider a nontrivial $\F(V)$-linear
relation among the $\om_i$ of minimal length, i.e.,
$\sum\limits_{j=1}^kf_j\om_{i_j}=0$ for a subset $\{i_1,\ldots,i_k\}$ of
$\{1,\ldots,m\}$ of minimal size and with coefficients $f_j\in \F(V)$. Without
loss of generality we may assume that $f_1=1$ and $f_2\in
\F(V)\setminus \F(V)^G$. Let $g \in G$ be an element for which $gf_2\neq
f_2$. Then
$$0=(1-g)(0)=(1-g)\left(\sum\limits_{j=1}^kf_j\om_{i_j}\right) = \sum
\limits_{j=2}^k (f_j-gf_j)\om_{i_j}$$
is a nontrivial relation (since $f_2-gf_2\neq 0$) of length $k-1$,
contradicting the minimality of $k$.

Next, we prove that the rank is at least $n$ by constructing $n$ linearly
independent invariant $1$-forms as follows. Let $z_1,\ldots,z_n$ be a basis of
$V^*$ and consider the union $U$ of the orbits of the $z_i$ under $G$. The Chern classes $c_i$ of
$U=\{u_1,\ldots,u_{r}\}$ are defined as the
coefficients of the polynomial 
$$\prod\limits_{i=1}^{r}(T-u_i) = T^{r} + c_1 T^{r-1} + ... + c_{r} \in \F[V]^G[T].$$

Consider the Jacobian matrix $J$ of the $c_i$ with respect to
$z_1,\ldots,z_n$. By the chain rule, we may write
$$J=\left(\frac{\partial c_i}{\partial u_j}\right)\cdot \left(\frac{\partial
  u_j}{\partial z_k}\right).$$
The first matrix is the Jacobian of the elementary symmetric polynomials $c_i$
  in the elements $u_j$ of $U$, which has nonzero determinant and thus full
  rank. The second matrix has rank $n$ since $\{z_1,\ldots,z_n\}\subseteq
  U$. Consequently, the rank of $J$ is $n$. This implies that we may choose $n$
  Chern classes $c_i$ for which the differential $1$-forms $dc_i$ are linearly independent over
  $\F[V]^G$ (since the wedge product of those $dc_i$ equals the nonzero determinant of
  the corresponding $n\times n$-minor of $J$ multiplied with
  $dz_1\wedge\cdots\wedge dz_n$).

\end{proof}

\section{Reflection groups and pointwise stabilizers}\label{setup}

An element of finite order in $\Gl(V)$ is a {\bf reflection} if
its fixed point space in $V$ is a hyperplane, 
called the {\bf reflecting hyperplane}.
There are two types of reflections: the diagonalizable reflections in $\Gl(V)$ have a single nonidentity eigenvalue which is a root of unity;  the nondiagonalizable reflections 
in $\Gl(V)$ are called {\bf transvections} and have determinant $1$ (note that they can only occur if the characteristic of $\F$ is positive).

Suppose $H\leq V$ is a hyperplane 
defined by a linear form $l_H$ in $V^*$ ($\ker l_H=H$), and let $s$ be a reflection about $H$.
Then there exists a vector $\alpha_s\in V$ for which  
$$ 
 s(v) = v + l_H(v)\, \alpha_s \ \ \quad\text{for all } v \in V,
$$
the {\bf root vector} of $s$ (with respect to $l_H$). 
Note that a transvection is a reflection whose root
vector (called a {\bf transvection root vector})
lies in its reflecting hyperplane, i.e., $l_H(\alpha_s) =0$.

A {\bf reflection group} is a finite group $G$ generated by reflections. 
The subgroup 
$$ 
  G_H:=\{g \in G:\; g{\vert}_{{\rule[0ex]{0ex}{1.5ex}}_H}=\operatorname{id}_H\}
$$ 
is the {\bf  pointwise stabilizer} of $H$ in $G$. The set of transvections in
$G_H$ together with the identity forms a normal subgroup $K_H$, the kernel of the determinant $\det:G_H\rightarrow \F^{\times}$.
Each $G_H$ is generated by $K_H$
together with a diagonalizable reflection $s_H$ of maximal order 
$e_H:=|G_H \colon\hspace{-.5ex}K_H|$. 
If none of the reflections about $H$ are diagonalizable, we set $s_H=1$.
In fact, each $G_H$ is isomorphic to a semi-direct product
$$
 G_H \cong  K_H\rtimes \Z/e_H \Z.
$$

The {\bf transvection root space} of $G_H$ is the 
subspace of $H$ generated by the root vectors of transvections in $G_H$. 
Let $b_H$ be its dimension.  
We remark that
$ b_H = \codim ((V^*)^{K_H})$, which can be seen by putting
the elements of $K_H$ into simultaneous upper triangular form. If $\F=\F_p$ is a prime field, then $|G_H|=e_H\ p^{b_H}$ (the formula becomes more complicated for larger fields, see \cite{hshepler}, Lemma~2.1).

\begin{lemma}
\label{orbits}
If $H$ and $H'$ are hyperplanes in the same $G$-orbit, then 
$e_H=e_{H'}$ and $b_H=b_{H'}$.
\end{lemma}
\begin{proof}
Let $H'=gH$ for some $g \in G$.
If $r$ is a reflection about $H$ and $g\in G$, then $grg^{-1}$ is a reflection
about $gH$. Consequently, $G_{H'}=G_{gH}=gG_Hg^{-1}$, and the claim follows.
\end{proof}

The following easy fact can be found, e.g., in \cite{kane}, 
Section~18.2. 
\begin{lemma}\label{delta}
If $s$ is a reflection about the hyperplane $H$ and $f$ is a polynomial, 
then $s(f)-f$ is divisible by $l_H$.
\end{lemma}

The next lemma is rather technical, but it is a key ingredient to the
freeness results in this article, and it highlights the importance of the
numbers $e_H$ and $b_H$. 
If $I=\{i_1,\ldots,i_m\}\subseteq \{1,\ldots,n\}$ is an ordered subset, and $z_1,\ldots,z_n\in V^*$, 
we write $dz_I$ for the product $dz_{i_1}\wedge \cdots\wedge dz_{i_m}$.
\begin{lemma}\label{divisions}
Suppose $H$ 
is a hyperplane defined by $l_H \in V^*$ and $G_H$ is a group of reflections about $H$. 
Let $K_H$ denote the kernel of the determinant character on $G_H$, and 
let $s_H$ in $G_H$ be a diagonalizable reflection of order $e_H$.
Let $v_1,\ldots,v_n$ be a basis of $V$ with the following properties: 
\begin{itemize}
\item $v_1,\ldots,v_{n-1}$ form a basis of $H$ 
\item $v_1,\ldots,v_{b_H}$ are $\F$-independent root vectors 
(with respect to $l_H$) of transvections in $K_H$, and
\item  $v_n\notin H$ is an eigenvector for $s_H$ with $l_H(v_n)=1$ .
\end{itemize}
(Such a basis always exists.)
Let $z_1,\ldots,z_n$ be the dual basis of $V^*$. Then
\begin{enumerate}
\item Let $\mu$ be a form invariant under $K_H$, and write $\mu=\sum\limits_{I} {u_I}\, dz_I$. Suppose that $J\cap \{1,\ldots,b_H\}\neq \varnothing$ and $n\notin J$. Then $l_H$ divides $u_J$.
\item Moreover, if $\mu$ is $G_H$-invariant, then $u_J$ is divisible by $l_H^{e_H}$.
\end{enumerate}
\end{lemma}
\begin{proof}
Let $J$ be as above, and let $m\in J\cap \{1,\ldots,b_H\}$.
Let $t_m\in K_H$ be a transvection with root vector $v_m$. 
Note that $t_m$ sends $dz_i$ to $dz_i$ for $i\neq m$
and $dz_m$ to $dz_m - dz_n$.  
In particular, $dz_I$ is invariant under $t_m$ if $I$ contains both $n$ and $m$.

Let $\sigma$ be the transposition switching $m$ and $n$.
Then
$$
 \begin{aligned}
 \sum_I u_I\ dz_I
 &=\mu = t_m(\mu) \\
 &= \sum_I t_m(u_I) t_m(dz_I)\\
 &=\sum_{I:\, m\notin I} t_m(u_I)dz_I
   +\sum_{I:\,m,\,n\in I} t_m(u_I) dz_I
   +\sum_{I:\, m\in I;\, n\notin I} t_m(u_I)
                                (dz_I\pm dz_{\sigma(I)})\\
 &=\sum_{I} t_m(u_I)dz_I
   +\sum_{I:\, m\in I;\, n\notin I} \pm t_m(u_I)\ 
                                dz_{\sigma(I)}\\
 &=\sum_{I} t_m(u_I)dz_I
   +\sum_{I:\, m\notin I;\,n\in I} \pm t_m(u_{\sigma(I)})
                                dz_{I}\\
 &=\sum_{I:\, m\in I\text{ or }n\notin I} t_m(u_I)dz_I
   +\sum_{I:\, m\notin I; n\in I} 
                      t_m(u_I) \pm t_m(u_{\sigma(I)}) 
                                dz_I.
 \end{aligned}
$$
Equating polynomial coefficients, we find that if $m\notin I$ and $n \in I$, then 
$$
  u_I=t_m(u_I) \pm t_m(u_{\sigma(I)})
  =t_m(u_I) \pm u_{\sigma(I)} \quad\text{and thus}\quad
  t_m(u_I)-u_I = \pm u_{\sigma(I)}
$$
(whereas $u_I$ is invariant in all other cases).
By Lemma~\ref{delta}, this implies that $u_J$ is divisible by $l_H$ 
(as $J= \sigma(I)$ for some such $I$). 

For the second statement, we may assume $G_H\neq K_H$, i.e.,
$s_H\neq 1$ (otherwise, we are in the situation
of (1)).  Note that $s:=s_H$ sends $dz_i$ to $dz_i$ for $i\neq n$
and $dz_n$ to $\lambda^{-1}dz_n$ where $\lambda:=\det(s)$.
Since $\mu$ is invariant under $s$,
$$
 \begin{aligned}
 \sum_I u_I\ dz_I
 &= \sum_{I:\ n\notin I} s(u_I)\ dz_I
    + \sum_{I:\ n\in I} \lambda^{-1} s(u_I)\ dz_I. \\
\end{aligned}
$$
Comparing coefficients shows that $u_J$ is invariant. But since $l_H$ divides $u_J$ and $\lambda$ has order $e_H$, $u_J$ must in fact be divisible by $l_H^{e_H}$.
\end{proof}

\section{A Criterion for Freeness of Invariant 1-Forms}\label{free}

Let $G$ be a finite group generated by reflections. Consider the collection
$\A=\A(G)$ of reflecting hyperplanes for $G$, the 
{\bf reflection arrangement} of $G$. 
For a linear character $\chi:G\rightarrow \F^{\times}$ of the group (acting on $V$)
and a $G$-module $M$, let $M^G:=\{m\in M: gm=m\}$ and $M^G_\chi:=\{m\in M:
gm=\chi(g) m\}$, the module of invariants and $\chi$-invariants (relative
invariants with respect to $\chi$),
respectively.
Let $\det:G\rightarrow \F^{\times}$
be the determinant character of $G$ (acting on $V$).
Define
$$
  Q_{\det}:=\prod_{H\in\A} l_H^{e_H-1},
$$
where $e_H$ is as defined in the last section. 

The following proposition is due to Stanley \cite{stanley} for $K=\C$. Nakajima
\cite{nakajima} proves a more general statement for arbitrary fields. A proof in
the flavor of Stanley's original argument can be found in Smith
\cite{larryspreprint}.
\begin{prop}\label{stanley}
Let $G$ be a reflection group. Then 
$$\F[V]^G_{\det} = \F[V]^G Q_{\det}.$$
\end{prop}
\noindent A similar statement is true for arbitrary linear characters of $G$, see 
Proposition~\ref{chistanley}.

If the characteristic of $\F$ is zero (or more generally, if the group $G$ does
not contain transvections), the group and its reflection arrangement can be
recovered from $Q_{\det}$ alone (as this polynomial encodes the orders
of the reflections about the hyperplanes as well as the hyperplanes 
themselves). Moreover, it detects generators for invariant
differential forms: The analogue of Saito's criterion for 
invariant 1-forms (this is a special case of a theorem of Orlik and Solomon, 
\cite{orliksolomon}, Theorem~3.1) asserts that invariant 1-forms $\om_1,\ldots,\om_n$ 
generate $(\Om^1)^G$ if $\om_1\wedge\cdots\wedge\om_n\doteq 
Q_{\det} dz_1\wedge\cdots\wedge dz_n$ 
for one and hence for any basis $z_i$ of $V^*$ (we write $a\doteq b$ to 
indicate that $a=cb$ for some $c\in \F^\times$).

However, in the general setting, $G$ contains transvections and 
$Q_{\det}$ does not carry enough information; we need another polynomial 
to encode characteristics of the transvection root space.
Let $\tilde{\A}$ be the multi-arrangement of hyperplanes formed
by assigning multiplicity $e_H\, b_H$ to each $H$ in $\A$.
Then $\tilde{\A}$ is defined by the polynomial
$$
  Q(\tilde{\A}) := \prod_{H\in \A} l_H^{\,e_H b_H}.
$$
Note that $Q(\tilde{\A})=1$ when all the reflections in $G$ are diagonalizable.

Fix a basis $z_1,\ldots,z_n$ for $V^*$ and
let $\vol$ be the volume form $dz_1\wedge\cdots\wedge dz_n$. 
Consider invariant 1-forms $\om_1,\ldots,\om_n$. Then
$\om_1\wedge\cdots\wedge\om_n=f\,\vol$ for
some polynomial $f \in \F[V]$. 
Since $\vol$ is $\det^{-1}$-invariant, $f$ must be 
$\det$-invariant. 
In particular, $Q_{\det}$ must divide $f$ by Proposition~\ref{stanley}. 
The analogue of Saito's criterion fails for groups which contain transvections
because $f$ is actually divisible by a polynomial of higher degree, 
$Q(\tilde{\A})Q_{\det}$.

\begin{lemma}\label{divides}
Suppose $G\leq \Gl_n(\F)$ is a reflection group.
If $\om_1,\ldots, \om_n$ are invariant $1$-forms,
then $Q(\tilde{\A})\, Q_{\det}$
divides $\om_1 \wedge\cdots\wedge \om_n$.
\end{lemma}
\begin{proof}
Fix a hyperplane $H \in \mathcal{A}$. We choose a basis of $V$ and $V^*$ as in
Lemma~\ref{divisions}. If $\mu$ is any invariant $1$-form and we write
$\mu=\sum_i u_idz_i$, then by the same lemma, the first $b_H$
coefficients $u_i$ are divisible by $l_H^{e_H}$. Consequently, the wedge product
of any $n$ invariant $1$-forms is divisible by $l_H^{e_Hb_H}$. Since the linear
forms defining different hyperplanes are relatively prime, $Q(\tilde{A})$
divides $\om_1\wedge\cdots\wedge\om_n$.

Because $Q(\tilde{A})$ is invariant, the quotient
$\om_1\wedge\cdots\wedge\om_n/Q(\tilde{A})$ is invariant and can be written
as $f \vol$ for some $\det$-invariant polynomial $f$. By Proposition~\ref{stanley}, $f$ is
divisible by $Q_{\det}$, proving the claim.
\end{proof}

The above lemma indicates the alteration needed for the criterion to
hold in all characteristics.

\begin{thm}
\label{solomon}
Suppose $G\leq \Gl_n(\F)$ is a reflection group.
Suppose $\om_1,\ldots, \om_n$ are invariant $1$-forms with 
$$
\om_1 \wedge\cdots\wedge \om_n \doteq Q(\tilde{\A})\, Q_{\det}\ \vol.
$$
Then $\om_1, \ldots, \om_n$ is a basis for
the set of invariant $1$-forms as a free module 
over the ring of invariants, $\F[V]^G$:
$$
(\Om^1)^G = \bigoplus_i \F[V]^G \ \om_i.
$$
\end{thm}
\begin{proof}
Since $\om_1\wedge\cdots\wedge\om_n$ is nonzero, the forms $\om_1, \ldots, \om_n$
are linearly independent over the field of fractions $\F(V)^G$ of $\F[V]^G$, and thus span $\F(V)^G\otimes_{\F[V]^G}(\Om^1)^G$ 
as a vector space over $\F(V)^G$.
Let $\om$ be an invariant $1$-form, and write $\om = \sum_i h_i \om_i$
with coefficients $h_i \in \F(V)^G$.  Fix some $i$ for which $h_i\neq 0$ and consider
$\om\wedge\om_{1}\wedge\cdots\wedge\om_{i-1}\wedge\om_{i+1}\wedge\cdots\wedge \om_n$.
Up to a nonzero scalar, this equals 
$$h_i\, \om_1\wedge\cdots\wedge\om_n=h_i\, Q(\tilde{\A})\, Q_{\det}\, \vol.$$ 
By Lemma~\ref{divides}, 
the product $\om\wedge\om_{1}\wedge\cdots\wedge\om_{i-1}\wedge\om_{i+1}\wedge\cdots\wedge \om_n$
is divisible by $Q(\tilde{\A})\, Q_{\det}$
and $h_i \in \F[V]\cap \F(V)^G=\F[V]^G$. 
\end{proof}

\section{An algebra structure on invariant differential forms}

Let $G$ be a finite reflection group. In this section, we explain how one can 
endow $\Om^G$ with the structure of a 
free exterior algebra when $(\Om^1)^G$ is a free $\F[V]^G$-module
and the transvection root space of each pointwise stabilizer is maximal.
(Here, each transvection root space has dimension $b_H = n-1$ and coincides with the hyperplane~$H$.)
We use a twisted wedge product to expose the free exterior algebra structure.
 
The lemma below holds for arbitrary finite subgroups of the general linear group, not just reflection groups (with $\A$ defined as the arrangement associated to the subgroup generated by reflections).
\begin{lemma}\label{wedge}
Let $G\leq \Gl_n(\F)$ be a finite group, and
suppose $\mu, \nu$ are $G$-invariant forms.  Then
$$
  \delta(\A_{n-1}) := \prod_{\substack{ H\in\A\\ b_H=n-1}} l_H^{e_H}
$$
divides $\mu\wedge\nu$.
\end{lemma}
\begin{proof}
Let $H$ be a reflecting hyperplane for which the transvection root space
of $G_H$ is maximal.
Fix a basis $z_1,\ldots,z_n$ as in the hypothesis of Lemma~\ref{divisions}.
Then by the same lemma, if $\varnothing \neq I\subseteq \{1,\ldots,n-1=b_H\}$ is an index set, $l_H^{e_H}=z_n^{e_H}$ must divide the coefficient 
$u_I$ in $\mu=\sum u_Idz_I$. A similar statement is true for $\nu$, and so $\mu\wedge\nu$ is divisible by $l_H^{e_H}$. 
The claim now follows from the fact that the linear forms defining different hyperplanes are relatively prime.
\end{proof}
\begin{remark}
The polynomial $\delta(\A_{n-1})$ may be interpreted as the {\bf discriminant polynomial} for the arrangement $\A_{n-1}:=\{H\in\A: b_H=n-1\}$
of hyperplanes with maximal transvection root spaces. In the theory of complex reflection groups, the discriminant polynomial is a product of linear forms defining the reflecting hyperplanes with each linear form raised to the power $e_H = |G_H|$, the maximal order of a (diagonalizable) reflection about the corresponding hyperplane.  It is an invariant polynomial of minimal degree which vanishes on the reflection arrangement.
\end{remark}

\begin{thm}\label{freealgebra}
Let $G\leq \Gl_n(\F)$ be a finite group which has polynomial invariants and suppose that 
the transvection root space of the pointwise stabilizer of any 
reflecting hyperplane is maximal.
If $\om_1, \ldots, \om_n$ are invariant $1$-forms with
$$\om_1\wedge\cdots\wedge\om_n \doteq Q(\tilde{\A})\, Q_{\det}\ \vol,$$
then they generate $\Om^G$
as a free exterior $\F[V]^G$-algebra
under the twisted wedge product
$$
 (\mu,\nu)\ \  \mapsto \ \ \frac{\mu\wedge\nu}{\delta(\A_{n-1})}\ .
$$
\end{thm}
\begin{proof}
Fix $k\in \{1,\ldots,n\}$. For each ordered index set $I=\{i_1,\ldots,i_k\}$ of length $k$ consider the k-form $$\om_I:=\frac{\om_{i_1}\wedge\cdots\wedge \om_{i_k}}{\delta(\A_{n-1})^{k-1}}\ ,$$ which is invariant by Lemma~\ref{wedge}. 
To prove the theorem, it suffices to show that these forms constitute a basis for $(\Om^k)^G$ as an $\F[V]^G$-module.
Since $\om_1\wedge\cdots\wedge\om_n$ is nonzero, the $\om_I$ are linearly 
independent over the field of fractions $\F(V)^G$ and thus form a basis  
of $\F(V)^G\otimes_{\F[V]^G}(\Om^k)^G$.

Let $\om$ be an invariant $k$-form and write $\om = \sum_I h_I \om_I$
with coefficients $h_I \in \F(V)^G$. We will show that these coefficients lie in $\F[V]^G$. To this end, fix some $I$ for which $h_I\neq 0$ and consider the complementary (ordered) index set $J=\{1,\ldots,n\}\setminus I$. Then 
\begin{equation*}
\begin{split}
\frac{\om\wedge\om_J}{\delta(\A_{n-1})}&= \frac{h_I\,\om_I\wedge\om_J}{\delta( \A_{n-1})}\\
&\doteq\frac{h_I\, \om_1\wedge\cdots\wedge\om_n}{\delta(\A_{n-1})^{1+|I|-1+|J|-1}}\\
&\doteq\frac{h_I \,Q(\tilde{\A})\, Q_{\det}}{\delta(\A_{n-1})^{n-1}}\ \vol .
\end{split}
\end{equation*}

By assumption, the arrangements $\A_{n-1}$ and $\A$ are the same for $G$ 
(as $b_H=n-1$ for each $H$), and thus
$$\delta(\A_{n-1})^{n-1} = \prod_{\substack{ H\in\A\\ b_H=n-1}} l_H^{e_H(n-1)}= \prod_{H\in\A} l_H^{b_He_H}=Q(\tilde{\A}).$$ 
Hence, 
$$\frac{\om\wedge\om_J}{\delta(\A_{n-1})}\doteq h_I\,Q_{\det}\ \vol,$$
 and the coefficient $h_I\,Q_{\det}$ is a $\det$-invariant polynomial. By Proposition~\ref{stanley}, it is divisible by $Q_{\det}$, 
which shows that $h_I\in \F[V]\cap\F(V)^G=\F[V]^G$.
\end{proof}

\section{Special Classes of Groups}\label{examples}
In this section, we explain how to obtain generating 1-forms from a set of basic invariants for three classes of reflections groups: groups fixing a single hyperplane pointwise, groups containing the special linear group, and unipotent groups. A
similar pattern can be seen in other examples of reflection
groups with polynomial rings of invariants (in fact, in all other examples
that we have examined). 
An interesting question is whether these examples are instances of a
more general phenomenon. Throughout this section, $\F=\F_{\!q}$ is a finite field of
characteristic~$p$.
\subsection{Pointwise Stabilizers of Hyperplanes}\label{one}
This subsection deals with groups of reflections about a single hyperplane. We show how to produce generating invariant 1-forms from the exterior derivatives of basic invariants as given in \cite{hshepler}, proof of Proposition~2.3. More precisely, we prove that those 1-forms are obtained by dividing by suitable powers of the linear form that defines the hyperplane under consideration: 
\begin{theorem}\label{onehyperplane}
Let $V$ be a vector space of dimension $n$ over a finite field $\F=\F_{\!q}$.
Let $G\leq \Gl (V)$ be a finite group which fixes a hyperplane $H\leq V$
pointwise, and let $e_H$ be the maximal order of a diagonalizable reflection in $G$. 
Let $l_H$ be a linear form defining $H$. Then there exist
basic invariants $f_1,\ldots,f_n$ of $G$ 
and natural numbers $a_i$ ($i=1,\ldots,n-1$) so that 
$$\frac{df_1}{(l_H^{e_H})^{a_1}},\ldots,\frac{df_{n-1}}{(l_H^{e_H})^{a_{n-1}}},\; df_n$$ generate $(\Om^1)^G$ as an $\F[V]^G$-module.
\end{theorem}
\begin{cor}
If $G\leq \Gl(V)$ is a finite group which fixes a hyperplane pointwise, then $(\Om^1)^G$ is a free $\F[V]^G$-module.
\end{cor}
The corollary is a direct consequence of the theorem since 
the 1-forms $df_1\ldots,df_n$ are linearly independent over $\F(V)^G$.

\begin{proof}[Proof of Theorem]
The invariant ring $\F[V]^G$ is a free polynomial algebra. We use the inductive
description of the basic invariants given in \cite{hshepler} (Proposition~2.3) and show that the theorem
holds in every step of the construction by applying Theorem~\ref{solomon}.

We begin by choosing a basis $z_1,\ldots,z_n$ of $V^*$ as in Lemma~\ref{divisions}.
Fix a set of generating elements $\{s, t_1,\ldots,t_r\}$ of $G$, where
$s$ is a diagonalizable reflection in $G$ of order $e:=e_H$, 
each $t_i$ is a transvection, and $r$ is minimal.
We successively consider the groups $G_i=\langle s, t_1,\ldots,t_i\rangle$.

There is nothing to prove for $G_0=\langle s \rangle$, since this is a
nonmodular group and we can choose all $a_i$ to be zero (cf.~\cite{benson}, Theorem 7.3.1). 

Suppose the theorem holds for $G_k$, and let $f_1,\ldots,f_n$
be basic invariants for $G_k$ with degrees $d_i$ and numbers $a_i$ as in the
statement, with $f_n=l_H^e$. 
By \cite{hshepler}, Proposition~2.3, we know that the degrees $d_i$ are $p$-powers for $i<n$
where $p$ is the characteristic of $\F$.

To construct a set of basic invariants for $G_{k+1}$, we relabel the $f_i$ as
follows: Among all $f_i$ of minimal degree not invariant under
$t_{k+1}$, we choose one with maximal number $a_i$ and label this polynomial $f_1$
(note that we refine the choice in the original procedure at this
point: a posteriori it will become apparent that in fact all the $f_i$ under consideration have the same $a_i$).

Define $f_n':=f_n$ (which
is invariant under $G_{k+1}$), and $a_n':=a_n=0$.
Define $f_2',\ldots, f_{n-1}'$ by $f_i':=f_i + c_i f_1^{\frac{d_i}{d_1}}$ where
the constants $c_i$ are chosen so that the $f_i'$ are invariant under $G_{k+1}$
(see \cite{hshepler}, proof of Proposition~2.3). 
Then
$$
 df_i'= df_i + c_i\frac{d_i}{d_1}f_1^{\frac{d_i}{d_1}-1}df_1.
$$
We record the change in the $a_i$: 
either $d_i=d_1$, in which case $df_i'$ is divisible by $f_n^{\operatorname{min}\{a_i,a_1\}}=f_n^{a_i}$, or $d_i>d_1$, which implies that $p$ divides $\frac{d_i}{d_1}$ and $df_i'=df_i$. In either case, we define $a_i':=a_i$ 
for $i\neq 1$ so that $(f_n')^{a_i'}=(l_H^{e})^{a_i'}$ divides $df_i'$.

We next take the product over the orbit of $f_1$ to produce a polynomial $f_1'$ invariant under $t_{k+1}$. Define 
$$h(X)=\prod\limits_{a\in A} (X+ a z_n^{d_1})\in {\mathbb F}[z_n][X],$$ where $A$ is a certain additive subgroup of ${\mathbb F}$ (defined in loc.\ cit.) of order $|{\mathbb F}_p(\lambda)|$ and $\lambda=\det(s)$. Let $m=(|A|-1)/e$. 
The polynomial $h(X)$ is additive and thus all exponents on $X$ in $h$ are $p$-powers. Let $f_1'=h(f_1)$. Then 
$$df_1' =  d(f_1c z_n^{d_1(|A|-1)})= c z_n^{d_1(|A|-1)}df_1 + f_1 c d_1 (|A|-1)z_n^{d_1(|A|-2)}dz_n,$$ where $c=\prod\limits_{a\in A\setminus\{0\}}a$. 
If $d_1\neq 1$, the second term is zero and we can set $a_1'=a_1+ d_1m$. If $d_1=1$, then $a_1=0$ and the highest power of $f_n$ dividing the new form is $f_n^{m-1}$;
hence we set $a_1'=m-1$.

In order to apply the criterion (Theorem~\ref{solomon}), we need to consider the product of the forms 
$\om_i:=df_i'/f_n^{a_i'}$:
\begin{equation*}
\begin{split}
\om_1\wedge\cdots\wedge \om_n
&=l_H^{-e\sum\limits_{i=1}^n a_i'}J(f_1',\ldots,f_n')\ \vol\\
&\doteq l_H^{-e\sum\limits_{i=1}^n a_i'} J(f_1,\ldots,f_n)z_n^{d_1(|A|-1)}\ \vol\\ 
&=l_H^{d_1(|A|-1)-e(a_1'-a_1)} Q(\tilde{\A}(G_k)) Q_{\det}\ \vol,
\end{split}
\end{equation*}
where $Q(\tilde{\A}(G_k))$ is the polynomial defining $\tilde{\A}$ for $G_k$ (a power of $l_H$) and
$J$ denotes the determinant of the Jacobian matrix
(see proof of Proposition~2.3 of \cite{hshepler}).

We consider two cases. The first case occurs when the dimension of the transvection root space of $G_k$ is the same as that of $G_{k+1}$. Since this dimension is $\operatorname{codim}((V^*)^{K_H})$, the groups $G_k$ and $G_{k+1}$ have the same number of {\em linear} invariants, and thus $d_1>1$. Then $Q(\tilde{\A}(G_k))=Q(\tilde{\A}(G_{k+1}))$, and 
$$\frac{l_H^{d_1(|A|-1)}}{l_H^{e(a_1'-a_1)}}=1,$$
 so the $\om_i$ satisfy the criterion. 
The second case occurs when the dimension of the transvection root space increases. In this case, $d_1=1$, 
and $Q(\tilde{\A}(G_{k+1}))=Q(\tilde{\A}(G_k))l_H^{\;e}$ by definition. Since
$$\frac{l_H^{d_1(|A|-1)}}{l_H^{e(a_1'-a_1)}}=\frac{l_H^{d_1(|A|-1)}}{l_H^{e(a_1')}}=l_H^{e},$$
the criterion is satisfied in this case as well.
\end{proof}
\vspace{2ex}
\begin{remark}
The point of Theorem~\ref{onehyperplane} is not primarily to provide generating forms. In fact, it is easy to see that the forms
$$
\begin{aligned}
 &z_n^{e_H}dz_1-z_1 z_n^{e_H-1}dz_n,\ 
 z_n^{e_H}dz_2-z_2 z_n^{e_H-1}dz_n,\ \ldots,\ 
 z_n^{e_H}dz_k-z_k z_n^{e_H-1}dz_n,\\ 
 &dz_{k+1}, \ldots, dz_{n-1},\
 z_n^{e_H-1} dz_n
\end{aligned}
$$
generate $(\Om^1)^G$, where $k:=b_H$ is the dimension of the transvection root space of~$G$. The theorem shows more: There exist basic invariants so that generators can be produced from their exterior derivatives. In fact, generators are found by dividing the exterior derivatives by (powers of) linear forms defining the reflecting hyperplane.

\end{remark}


\subsection{Groups containing the special linear group}
We turn to the case when the finite group $G$ contains
$\Sl_n(\F_{\!q})$, i.e., $\Sl_n(\F_{\!q})\leq G\leq \Gl_n(\F_{\!q})$
for a finite field $\F_{\!q}$.
Such groups are parametrized by the order $e$ of their image under the
determinant homomorphism $\det:\Gl_n(\F_{\!q})\rightarrow \F_{\!q}^\times$. Note that
all these groups are generated by reflections (those generating $\Sl_n$ plus a
diagonalizable reflection with eigenvalue of order $e$), and all of them have a
polynomial ring of invariants (see for example, \cite{Poly}, comment after
Theorem~8.1.8). With this example, we illustrate the notation used in this paper as 
well as some of the results.

We summarize the setup: Since $\Sl_n(\F_{\!q})\leq G$, every hyperplane in $V$ is a reflecting hyperplane, and there are $\frac{q^n-1}{q-1}$ such hyperplanes, since there are $q^n-1$ nonzero elements in $V^*$, $q-1$ of which are nonzero scalar multiples of any fixed one (and thus define the same hyperplane). This describes the reflection arrangement $\A$. The multi-arrangement $\tilde{\A}$ is defined via the numbers $e_H$ and $b_H$ for each hyperplane $H$. All hyperplanes are in the same $G$-orbit, so in fact $e_H$ and $b_H$ do not depend on $H$. Since $G$ contains $\Sl_n(\F_{\!q})$, $b_H=n-1$ for every $H$ and all transvection root spaces are maximal. Moreover, $e_H=e$, the order of the image of $G$ under the determinant homomorphism. Consequently, $Q(\tilde{\A}) = \prod\limits_{H\leq V} l_H^{(n-1)e}$, and its degree is $\frac{(q^n-1)e(n-1)}{q-1}$.

We first describe generators for the module of $1$-forms invariant under the full
group $\Gl_n(\F_{\!q})$.
 The ring of invariant polynomials 
$\F[V]^{\Gl_n(\F_{\!q})}$ is called the {\bf Dickson algebra}, and the {\bf Dickson invariants} 
$$
d_{n,i} \ := \ \sum_{
         \substack{W \leq V\\ \rule[.7ex]{0ex}{1ex}\codim W \, = \, i}} 
         \ \ \prod_{\substack{v \, \in \, V^*,\\ v|_W\neq 0}} v
$$
of degree $q^n-q^i$ (for $i=0,\ldots,n-1$)
form a set of basic invariants.
The determinant of the Jacobian matrix of the Dickson invariants
is $$J = \prod_{H\in\A} l_H^{(n-1)(q-1)+(q-2)}$$ (see \cite{hshepler}, Section 4).
Hence, by Theorem~\ref{solomon}, $(\Om^1)^{\Gl_n(\F_{\!q})}$ is a free $\F_{\!q}[V]^{\Gl_n(\F_{\!q})}$-module generated by the exterior derivatives of the Dickson invariants.

We use the exterior derivatives of the Dickson invariants to 
construct generators for 
$(\Om^1)^G$. Consider a hyperplane $H\leq V$, and choose a basis $z_1,\ldots,z_n$ for $V^*$ as in Lemma~\ref{divisions}. The polynomials 
$$f_1^H:=z_1^q-z_1z_n^{q-1},\quad\ldots,\quad f_{n-1}^H:=z_{n-1}^q-z_{n-1}z_n^{q-1},\quad f_n^H:=z_n^{q-1}$$  
are invariant under the pointwise stabilizer $\Gl_n(\F_{\!q})_H$ and algebraically independent. The product of their degrees equals $(q-1)\cdot q^{n-1}=|\Gl_n(\F_{\!q})_H|$, so by \cite{kemper}, Theorem~7.3.5., $$\F[V]^{\Gl_n(\F_{\!q})_H}=\F[f_1^H,\ldots,f_n^H].$$
Consequently, there are polynomials $p_i$ for which $d_{n,i}=p_i(f_1^H,\ldots,f_n^H)$. Then 
$$\frac{\partial d_{n,i}}{\partial z_j}=\sum \limits_{k=1}^n \frac{\partial d_{n,i}}{\partial f_k^H}\frac{\partial f_k^H}{\partial z_j}$$
by the chain rule. Since $\operatorname{char}(\F)=p>0$, $\frac{\partial f_k^H}{\partial z_j}$ is divisible by $z_n^{q-2}=l_H^{q-2}$ (for all $k,j$), and the same is true for $d(d_{n,i})$. In particular, each $d(d_{n,i})$ is divisible by $f:=\prod\limits_{H\leq V} l_H^{q-e-1}$, which is invariant under $G$. 
We may therefore define forms 
$$\om_i:=\frac{d(d_{n,i-1})}{f}\in \Om^G.$$
Their wedge product is
\begin{equation*}
\begin{split}
\om_1\wedge\cdots\wedge\om_n
&=d(d_{n,0})\wedge\cdots\wedge d(d_{n,n-1})\cdot f^{-n}\\
&=\frac{J}{f^n}\ \vol\\
&\doteq \prod_{H\in\A}l_H^{(n-1)(q-1)+(q-2) -n(q-e-1)}\ \vol
=\prod_{H\in\A}l_H^{ne-1}\ \vol\\
&=\prod_{H\in\A}l_H^{(n-1)e +(e-1)}\ \vol
=Q(\tilde{\A})Q_{\det}\ \vol .
\end{split}
\end{equation*}
Hence, by Theorem~\ref{solomon}, $\om_1, \ldots, \om_n$ generate $(\Om^1)^G$ as a free
$\F[V]^G$-module.

Moreover, Theorem~\ref{freealgebra} implies that $\Om^G$ is a free algebra under the twisted wedging 
$$(\mu,\nu)\mapsto \frac{\mu\wedge\nu}{\delta(\A_{n-1})}= \frac{\mu\wedge\nu}{\prod\limits_{H\leq V} l_H^e}\ .$$

In the case $G=\Gl_n(\F_{\!p})$, generators for the invariant forms as a module over the Dickson algebra were given by Mui (\cite{mui}) in terms of Vandermonde-like determinants. 
He also lists the relations among those generators under the usual wedging. The above calculation simplifies his approach.

\begin{remark}
Alternatively, for $\F=\F_{\!p}$, one may start with generators of the set
of 1-forms invariant under $\Sl_n(\F_{\!p})$ listed in \cite{Adem}, Definition III.2.8 (following Mui's work), which
are $\det^{-1}$-invariant under $\Gl_n(\F_{\!p})$.  Multiplying these by $Q_{\det}$
produces forms invariant under $G$.  A calculation similar to the one above then shows that they generate $(\Om^1)^G$ as a module, and more generally, $\Om^G$ as a free algebra under the twisted wedging. 
\end{remark}

\begin{remark}
The pattern in the above example holds in greater generality
when $\F=\F_{\!q}$ is a finite field.
Suppose $G\leq\Gl(V)$ has basic invariants $f_1, \ldots, f_n$ and every reflecting
hyperplane for $G$ has a diagonalizable reflection of maximal order $q-1$ 
(as in the case $G=\Gl_n(\F_{\!q})$).
One can then prove that $(\Om^1)^G$ is a free $\F[V]^G$-module generated by
$df_1, \ldots, df_n$.
In fact, if
$G'\leq G$ is any subgroup with polynomial invariants sharing the same reflection arrangement,
$(\Om^1)^{G'}$ is a free $\F[V]^{G'}$-module provided each transvection root space for $G'$ is maximal (as in the case $\Sl_n(\F_{\!q}) \leq G' \leq \Gl_n(\F_{\!q})$).
As above, dividing each $df_i$ by $\prod\limits_{H\in \A}l_H^{q-e_H-1}$ produces generators for $(\Om^1)^{G'}$ from basic invariants for $G$.
\end{remark}

\subsection{Unipotent Groups}
Another example is given by the unipotent group $G=\operatorname{U}_n(\F_{\!q})$,
 which we regard in its 
representation as lower triangular matrices of determinant~1.

Consider the polynomials
$$
\begin{aligned}
f_1&:=z_1\\
f_2&:=\prod\limits_{\alpha_1 \in \F_{\!q}}( z_2 + \alpha_1 z_1)\\
\vdots\\
f_n&:=\prod\limits_{(\alpha_1,\ldots,\alpha_{n-1}) \in \F_{\!q}^{n-1}} (z_n + \alpha_{n-1} z_{n-1} + \cdots + \alpha_1 z_1).
\end{aligned}
$$
These polynomials are invariant under $G$, algebraically independent, and have degrees $1$, $q$, $\ldots$, $q^{n-1}$. Since the product of these degrees is exactly the order of $G=\operatorname{U}_n(\F_{\!q})$, they form a set of basic invariants by \cite{kemper}, Theorem 7.3.5.

\begin{remark}\label{remark}
Note that $f_i=\prod\limits_{\substack{H\in\A\rule[-.5ex]{0ex}{2ex} \\
    b_H=n-i}}l_H$.  In fact, $f_i$ 
is the product over the $G$-orbit of any fixed $l_H$ with $b_H=n-i$. 
\end{remark}

\begin{lemma}\label{unipotents}
With notation as above, $df_k$ is divisible by $\prod\limits_{i<k}f_i^{q-2}$ for $k=1,\ldots,n$. 
\end{lemma}
\begin{proof}
Because each $df_k$ is invariant, and because $f_i$ is the product over the orbit of $z_i$, it suffices to show that $z_i^{q-2}$ divides $df_k$ for each $i<k$. We show that $z_i^{q-2}$ divides $\frac{\partial f_k}{\partial z_l}$ for each $i<k$ and each $l=1,\ldots,n$.

For $i<k$, rewrite
\begin{equation*}\begin{split}
f_k&=\prod\limits_{(\alpha_1,\ldots,\alpha_{k-1}) \in \F_{\!q}^{k-1}} (z_k + \alpha_{k-1} z_{k-1} + \cdots + \alpha_1 z_1)\\&=\prod\limits_{\bar{\alpha} = (\alpha_1,\ldots,\widehat{\alpha_i},\ldots,\alpha_{k-1}) \in \F_{\!q}^{k-2}}\;\;
\prod\limits_{\alpha_i\in\F_{\!q}}(z_k + \alpha_{k-1} z_{k-1} + \cdots + \alpha_1 z_1)\\
&=\prod\limits_{\bar{\alpha} \in \F_{\!q}^{k-2}}\;\;\prod\limits_{\alpha_i\in\F_{\!q}}(z_{\bar{\alpha}} + \alpha_iz_i)\\
&= \prod\limits_{\bar{\alpha} \in \F_{\!q}^{k-2}} (z_{\bar{\alpha}}^q - z_i^{q-1}z_{\bar{\alpha}}),
\end{split}
\end{equation*}
where $z_{\bar{\alpha}}=z_k + \alpha_{k-1}z_{k-1} + \cdots + \widehat{\alpha_iz_i}+\cdots +\alpha_1z_1$.
Consequently, 
\begin{equation*}\begin{split}
\frac{\partial f_k}{\partial z_l}
& = \sum \limits_{\bar{\alpha}\in\F_{\!q}^{k-2}} \left(
\frac{\partial}{\partial z_l}(z_{\bar{\alpha}}^q-z_i^{q-1}z_{\bar{\alpha}})
\prod\limits_{\bar{\beta}\neq \bar{\alpha}} (z_{\bar{\beta}}^q - z_i^{q-1}z_{\bar{\beta}})\right)\\
&=\sum \limits_{\bar{\alpha}\in\F_{\!q}^{k-2}}\left(\left(z_i^{q-2}\delta_{il}z_{\bar{\alpha}}-z_i^{q-1}\frac{\partial z_{\bar{\alpha}}}{\partial z_l}\right)\prod\limits_{\bar{\beta}\neq \bar{\alpha}} (z_{\bar{\beta}}^q - z_i^{q-1}z_{\bar{\beta}})\right)
\end{split}\end{equation*}
which is divisible by $z_i^{q-2}$ as claimed ($\delta_{il}$ is the Kronecker delta symbol). 
\end{proof}

Moreover, the calculation in the above proof shows that $\frac{\partial f_k}{\partial z_k}$ is divisible by $z_i^{q-1}$ (since the term involving $\delta_{ik}$ is zero), and thus by $\prod\limits_{i<k}f_i^{q-1}$. 
The degree of this product is $(q-1)(1+q+\cdots +q^{k-1})=q^k-1$, 
which is also the degree of $\frac{\partial f_k}{\partial z_k}$, so $\frac{\partial f_k}{\partial z_k}\doteq \prod\limits_{i<k}f_i^{q-1}$.

We assert that the 1-forms 
$\displaystyle{\om_k:=\frac{df_k}{\prod\limits_{i<k}f_i^{q-2}}
}$ 
generate $(\Om^1)^G$.
To see this, consider the Jacobian matrix of the $f_i$. Note that $\frac{\partial f_k}{\partial z_l}=0$ if $l>k$, so this is a lower triangular matrix, and its determinant $J$ is the product of its diagonal entries:
$$J=\prod\limits_{k=1}^n \frac{\partial f_k}{\partial z_k} = \left(
\prod\limits_{k=1}^n\;\;\prod\limits_{i<k}f_i\right)^{q-1}.$$

Consequently, 
\begin{equation*}
\begin{split}
\om_1\wedge\cdots\wedge\om_n
&=
\left(\prod\limits_{k=1}^n\;\;\prod\limits_{i<k}f_i\right)^{2-q}
df_1\wedge\cdots\wedge df_n \\
&\doteq 
\left(\prod\limits_{k=1}^n\;\;\prod\limits_{i<k}f_i\right)^{2-q}\ J \ \vol\\
&=\left(\prod\limits_{k=1}^n\;\;\prod\limits_{i<k}f_i\right) \ \vol\\
&=\prod\limits_{i=1}^n f_i^{n-i}\ \vol \\
&=\prod\limits_{H \in \A} l_H^{b_H}\ \vol\quad\text{ by Remark~\ref{remark}}\\ 
&=Q(\tilde{\A})\ \vol=Q(\tilde{\A})Q_{\det}\ \vol
\end{split}
\end{equation*}
since $e_H=1$ for all $H$, and the claim follows by Theorem~\ref{freealgebra}.

For $\F=\F_{\!p}$, another description of this module of invariant forms 
can be found in \cite{mui}.


\section{Invariants Relative to a Character}
\label{chi}

The results in Section~\ref{free} have generalizations to relative invariants with respect to a linear character of the reflection group $G$. We first define the corresponding arrangements and polynomials.
For any linear character $\chi$ of $G$, define
$$
 Q_{\chi}:= \prod_{H\in\A} l_H^{a_H},
$$
where $a_H$ is the smallest nonnegative integer satisfying
$\chi(s_H)= \det^{-a_H}(s_H)$.
Stanley proved \cite{stanley} the following analogue of Proposition~\ref{stanley} for
complex reflection groups 
(again, the proof extends to arbitrary characteristic, see the remarks in Section~\ref{free}).
\begin{prop}\label{chistanley}
If $G$ is a reflection group,
then $\F[V]^G_\chi = \F[V]^G Q_\chi$. 
\end{prop}

We next define a $\chi$-version of the multi-arrangement $\tilde{\A}$.
Let $\tilde{\A}_\chi$ be the multi-arrangement defined by the polynomial
$$
  Q(\tilde{\A}_\chi)
   = \prod_{\substack{H\in\A\rule[-.5ex]{0ex}{2ex} \\ \chi(s_H)=1}} 
             l_H^{\,e_H b_H}.
$$
The following generalizes results from Shepler~\cite{shepler}.
\begin{lemma}\label{chidivides}
Suppose $G\leq \Gl_n(\F)$
is a reflection group and let $\chi$ be a linear character of $G$.
If $\om_1,\ldots, \om_n$ are $\chi$-invariant $1$-forms,
then 
$Q(\tilde{\A}_\chi)
\, Q_\chi^{n-1}\, Q_{\chi\cdot\det}$
divides $\om_1 \wedge\cdots\wedge \om_n$.
\end{lemma}
\begin{proof}
Fix some reflecting hyperplane $H \in \A$ with diagonalizable reflection 
$s:=s_H$ of (maximal) order $e_H$.
Choose a basis of $V$ and $V^*$ as in Lemma~\ref{divisions}. 
Then $s$ sends $dz_i$ to $dz_i$ for $i\neq n$
and $dz_n$ to $\lambda^{-1}dz_n$ where $\lambda:=\det(s)$.

Let $\mu$ be any $\chi$-invariant $1$-form and write $\mu=\sum_{i} u_i \ dz_i$.
As before, let $K_H$ be the set of elements of determinant~1 in $G_H$. 
Since $\chi(K_H)=1$, $\mu$ is $K_H$-invariant and thus by the first part of Lemma~\ref{divisions}, 
$z_n$ divides the first $b_H$ coefficients of $\mu$.

We now consider the action of the diagonal group element $s$ on $\mu$.
Since $\mu$ is $\chi$-invariant and $\lambda^{-a_H}=\chi(s)$,
$$
 \begin{aligned}
 \lambda^{-a_H}\sum_i u_i\ dz_i = \chi(s) \mu
 &= s(\mu) = \sum_{i \neq n} s(u_i)\ dz_i
    +  \lambda^{-1} s(u_n)\ dz_n. \\
\end{aligned}
$$
Consider the $i$-th coefficient, with $i<n$.
Then $s(u_i)=\lambda^{-a_H}u_i$ and hence $z_n^{a_H}$ divides $u_i$.
Thus
the first $n-1$ coefficients of $\mu$
are divisible by $z_n^{a_H}$.
And when $\chi(s)=1$, more is true:
the first $b_H$ coefficients of
$\mu$ are in fact divisible by $z_n^{e_H}$ by the second part of Lemma~\ref{divisions}.
(Recall that $0\leq a_H<e_H$.)

Thus, $\om_1\wedge\cdots\wedge\om_n$
is divisible by
$l_H^{a_H(n-1)}$. And when $\chi(s)=1$,
then $\om_1\wedge\cdots\wedge\om_n$
is divisible by $l_H^{e_H b_H}$.
Since $Q(\tilde{\A}_\chi)$ and $Q_\chi^{n-1}$ have no common factors, it follows that 
$Q(\tilde{\A}_\chi)\, Q_\chi^{n-1}$ divides $\om_1\wedge\cdots\wedge\om_n$.

Because $Q(\tilde{\A}_\chi)$ is invariant and $Q_\chi$ is $\chi$-invariant,
$\om_1\wedge\cdots\wedge\om_n\,(Q(\tilde{\A}_\chi)\, Q_\chi^{n-1})^{-1} $ 
is a $\chi$-invariant $n$-form, i.e., equals $f \vol$ for some 
$(\chi\det)$-invariant polynomial $f$ (as $\vol$ is $\det^{-1}$-invariant).
By Proposition~\ref{chistanley}, $f$ is divisible by $Q_{\chi\det}$.  
So $\om_1\wedge\cdots\wedge\om_n$
is in fact divisible by $Q(\tilde{\A}_\chi)\, Q_\chi^{n-1}\, Q_{\chi\det}$.
\end{proof}
\begin{thm}
\label{chisolomon}
Suppose $G\leq \Gl_n(\F)$ is a reflection group.
Let $\chi$ be a linear character of $G$.
Suppose $\om_1,\ldots, \om_n$ are $\chi$-invariant $1$-forms with 
$$
\om_1 \wedge\cdots\wedge \om_n \doteq Q(\tilde{\A}_\chi)\, Q_\chi^{n-1}\, Q_{\chi\det}\ \vol.
$$
Then $\om_1, \ldots, \om_n$ is a basis for
the set of $\chi$-invariant $1$-forms as a free-module 
over the ring of invariants, $\F[V]^G$:
$$
(\Om^1)_\chi^G = \bigoplus_i \F[V]^G \ \om_i.
$$
\end{thm}
\begin{proof}
The proof is completely analogous to the proof of Theorem~\ref{solomon}:
Since $\om_1\wedge\cdots\wedge\om_n$ is nonzero, the forms $\om_1, \ldots, \om_n$
are linearly independent over the field of fractions $\F(V)^G$ of $\F[V]^G$, and thus span $\F(V)^G\otimes_{\F[V]^G}(\Om^1)_\chi^G$ 
as a vector space over $\F(V)^G$.
Let $\om$ be a $\chi$-invariant $1$-form, and write $\om = \sum_i h_i \om_i$
with coefficients $h_i \in \F(V)^G$.  Fix some $i$ for which $h_i\neq 0$ and consider
$\om\wedge\om_{1}\wedge\cdots\wedge\om_{i-1}\wedge\om_{i+1}\wedge\cdots\wedge \om_n$.
Up to a nonzero scalar, this equals 
$$h_i\, \om_1\wedge\cdots\wedge\om_n\doteq h_i\,  Q(\tilde{\A}_\chi)\, Q_\chi^{n-1}\, Q_{\chi\det}\, \vol.$$ 
By Lemma~\ref{chidivides} above, the product is divisible by $Q(\tilde{\A}_\chi)\, Q_\chi^{n-1}\, Q_{\chi\det}$, 
i.e., $h_i \in \F[V]^G$. 
\end{proof}

\bibliography{HartmannSheplerArxiv}

\end{document}